\documentclass[11pt]{amsart}
\usepackage{ amsmath}
\usepackage{amssymb}
\usepackage{amsxtra}
\usepackage[mathscr]{eucal}
\usepackage[margin=3.4cm]{geometry}

\def\R{\mathbb R}

\def\H{\mathbb H}
\def \h{\bf H}

\def\C{\mathbb C}

\def\s{\mathbb S}

\def\S{\rm SL(2,  \mathbb H)}
\def \s{\mathbb S}
\newtheorem{theorem}{Theorem}[section]
\newtheorem{lemma}[theorem]{Lemma}

\theoremstyle{definition}

\theoremstyle{remark}
\newtheorem{remark}[theorem]{Remark}
\numberwithin{equation}{section}
\theoremstyle{plain}

\newtheorem{cor}[theorem]{Corollary}

\newcommand{\secref}[1]{Section~\ref{#1}}
\newcommand{\thmref}[1]{Theorem~\ref{#1}}

\newcommand{\corref}[1]{Corollary~\ref{#1}}
\newcommand{\eqnref}[1]{~{\textrm(\ref{#1})}}

\begin{document}
\title[ Quaternionic J\o{}rgensen Inequality]{ Extremality of Quaternionic J\o{}rgensen Inequality}
\author[Krishnendu Gongopadhyay \and Abhishek Mukherjee]{Krishnendu Gongopadhyay \and Abhishek Mukherjee}
 \address{Department of Mathematical Sciences, Indian Institute of Science Education and Research (IISER) Mohali,
Knowledge City, Sector 81, S.A.S. Nagar, P.O. Manauli 140306, India}
\email{krishnendu@iisermohali.ac.in, krishnendug@gmail.com}
\address{ Kalna College, Kalna, Dist. Burdwan, West Bengal 713409, India}
\email{abhimukherjee.math10@gmail.com }


\thanks{ }

\date{\today}

\begin{abstract}
Let $\S$  be the group of $2 \times 2$ quaternionic matrices with Dieudonn\'e determinant $1$.  The group $\S$ acts on the five dimensional hyperbolic space by isometries.  We ask extremality of J\o{}rgensen type inequalities in $\S$. Along the way, we derive  J\o{}rgensen type inequalities for quaternionic M\"obius transformations which extend earlier inequalities obtained by Waterman and Kellerhals. 
\end{abstract}
\maketitle

\section{Introduction}
 In the theory of Fuchsian groups, one of the important old problem is the ``\hbox{discreteness} problem": given two elements in ${\rm PSL}(2, \R)$, whether or not the group generated by them is discrete. For an elaborate account of this problem, see Gilman \cite{gilman}. \hbox{Algorithmic} solutions to this problem were given by Rosenberger \cite{r}, Gilman and Maskit \cite{gm}, Gilman \cite{gilman}.  The J\o{}rgensen inequality \cite{j} is one of the major results related to this problem.  J\o{}rgensen \cite{j} obtained an inequality that the generators of a discrete, non-elementary, two-generator subgroup of ${\rm SL}(2, \C)$ necessarily satisfy. Wada \cite{w} used this inequality to provide an effective algorithm that helps the software OPTi to test discreteness of subgroups, as well as to draw deformation spaces of discrete groups. 

A two-generator discrete subgroup of isometries of the hyperbolic space is called \emph{extreme group} if it satisfies equality in the J\o{}rgensen inequality.  Investigation of extreme groups in ${\rm SL}(2, \C)$ was initiated by J\o{}rgensen and Kikka \cite{jk}. Following that, there have been many investigations to  classify the two-generator extreme groups in ${\rm SL}(2, \C)$, for eg. see \cite{gm, gr}. In a series of papers, Sato et. al.  \cite{sato0}--\cite{sato5} have investigated this problem in great detail and provided a conjectural list of the parabolic-type extreme groups. Callahan \cite{cal} has provided a counter example to that conjecture. Callahan has also classified all non-compact arithmetic extreme groups that was not in the list of Sato et. al.  The problem of classifying parabolic-type J\o{}rgensen groups in ${\rm SL}(2, \C)$ is still open. Recently, Vesnin and Masley \cite{vm} have investigated extremality of other J\o{}rgensen type inequalities  in ${\rm SL}(2, \C)$.

The problem of classifying extreme J\o{}rgensen groups in higher dimension has not seen much investigation till date. The aim of this paper is to address this problem for J\o{}rgensen type inequalities in $\S$, where $\H$ is the division ring of the real quaternions and  $\S$ is the group of $2 \times 2$ quaternionic matrices with Dieudonn\'e determinant $1$. It is well-known that $\S$ acts on the  five dimensional  real  hyperbolic space $\h^5$ by the M\"obius transformations (or linear fractional transformations), for a proof see \cite{kg}. The isometries of $\h^5$ are classified by their fixed points as elliptic, parabolic and hyperbolic (or loxodromic). This classification can be characterized algebraically by conjugacy invariants of the isometries, see \cite{p, ps, kg, cao} for more details.

 The J\o{}rgensen inequality has been generalized in higher dimensions by Martin \cite{martin} who formulated it by identifying the hyperbolic space as the upper half space or the unit ball in $\R^{n+1}$. Hence, in Martin's generalization,  the isometries are real matrices of rank $n+1$. Generalizing the approach of using rank two real and complex matrices in low dimensions,  Ahlfors \cite{ahlfors} used Clifford algebras to investigate higher dimensional M\"obius groups. In this approach, the  isometry group of the hyperbolic $n$-space can be identified with a group of $2 \times 2$ matrices over the Clifford numbers, see Ahlfors \cite{ahlfors}, Waterman \cite{waterman} for more details. Using the Clifford algebraic formalism, a generalization of J\o{}rgensen inequality was obtained by Waterman \cite{waterman}.  However, it may be difficult to deal with the Clifford matrices due to the complicated multiplicative structure of the Clifford numbers. 

Using the real quaternions there is an intermediate approach between the complex numbers and the Clifford numbers, that should provide the  closest generalization of the low dimensional results for four and five dimensional M\"obius groups.    The Clifford group that acts by isometries on the hyperbolic $4$-space, is a proper subgroup of $\S$. So, Waterman's result restricts to this case.  Kellerhals \cite{kel2} has used this quaternionic Clifford group to investigate collars in $\h^4$. Recently, Tan et. al. \cite{t} have obtained a generalization of the classical Delambre-Gauss formula for right-angles hexagons in hyperbolic $4$-space using the quaternionic Clifford group of Ahlfors and Waterman.

 The Clifford group that acts on $\h^5$, however, is not a subgroup of $\S$. In fact, the group $\S$ is not in the list of the Clifford groups of Ahlfors and Waterman. However, following the approaches of Waterman, it is not hard to formulte J\o{}rgensen type inequalities for pairs of isometries in $\S$. Kellerhals \cite{kel} derived J\o{}rgensen inequality for two-generator discrete subgroups in $\S$ where one the of the generators is either unipotent parabolic or hyperbolic.   

Using similar methods as that of Waterman, we give here slightly generalized versions of the J\o{}rgensen inequalities in $\S$  when one of the generators is either semisimple or fixes a point on the boundary, see \thmref{jss} and \thmref{jg} in \secref{jse}. As corollaries we derive the formulations by Kellerhals and Waterman in the quaternionic set up, see \corref{kele} and  \corref{wat} respectively.   We also formulate a  J\o{}rgensen type inequality for strictly hyperbolic elements that  is very close to the original formulation of J\o{}rgensen, see  \corref{jh}. We recall here that a strictly hyperbolic element or a stretch is conjugate to a diagonal matrix that has real diagonal entries different from $0, ~1$ or $-1$. We also give as corollaries two weaker versions of the inequality when one generator is semisimple. 

We investigate the extremality of these J\o{}rgensen inequalities in \secref{sext}. We extend the results of J\o{}rgensen and Kikka in the quaternionic set up, see \thmref{ext1}, \corref{extc1} and, \thmref{extt2}. We also obtain necessary conditions for a two-generator subgroup of $\S$ to be extremal, see \corref{extp1} and \corref{extc2}.

\section{Preliminaries}
\subsection{The Quaternions} Let $\H$ denote the division ring of quaternions. Recall that every element of $\H$ is of the form  $a_{0}+a_{1}i+a_{2}j+a_{3}k$,where $a_{0},a_{1},a_{2},a_{3}\in \R$, and  $ i,j,k$ satisfy relations: $i^{2}=j^{2}=k^{2}=-1,ij=-ji=k,jk=-kj=i,ki=-ik=j$, and $ ijk=-1$. Any $a\in {\H}$ can be written as  $a=a_{0}+a_{1}i+a_{2}j+a_{3}k=(a_{0}+a_{1}i)+(a_{2}+a_{3}i)j=z+wj$, where $z=a_{0}+a_{1}i,~ w=a_{2}+a_{3}i\in\bf{\C}$. For $a\in\bf{\H}$,with $a=a_{0}+a_{1}i+a_{2}j+a_{3}k$,we define $\Re(a)=a_{0}$=the real part of $a$ and $\Im(a)=a_{1}i+a_{2}j+a_{3}k=$ the imaginary part of $a$. Also,define the conjugate of $a$ as $\overline {a}= \Re(a)-\Im(a)$
If $\Re(a)=0$,then we call $a$ as a vector in $\H$ which we can identify with ${\R}^{3}$. The norm of $a$ is $|a|=\sqrt{a_0^{2}+a_1^{2}+a_2^{2}+a_3^{2}}$.
\subsubsection{{Useful Properties}}
We note the following properties of the quaternions that will help us further:
\begin{enumerate}
\item  {For $x\in{\bf{\R}},~ a\in{\H},\text{ we have}\medspace ax=xa $.}
\item {For $ a\in{\bf{\C}},~ aj=j\overline{a}$.}
\item {For $a,b\in{\H},|ab|=|a||b|=|ba|\thickspace \text{and if}\thickspace a\neq 0,\text{then}\thickspace a^{-1}=\frac{\overline{a}}{|a|^2}$.}
\end{enumerate}

Two quaternions ${a,b}$ are said to be \emph{similar} if there exists a non-zero quaternion $ {c}$ such that $ {b=c^{-1}ac}$ and we write it as $ {a\backsim b}$. Obviously $ {'\backsim'}$  is an equivalence relation on ${\H}$ and denote $[a]$ as the class of $a$. It is easy to verify that $ {a \backsim b}$ if and only if $ {\Re(a)=\Re(b)}$ and $|a|=|b|$. Equivalently, $ {a \backsim b}$ if and only if $ {\Re(a)=\Re(b)}$ and $|\Im (a)|=|\Im (b)|$. Thus the similarity class of every quaternion $a$ contains a pair of complex conjugates with absolute-value $|a|$ and real part equal to $\Re( a)$.  Let $a$ is similar to $re^{i \theta}$, $\theta \in [-\pi, \pi]$. In most cases, we will adopt the convention of calling $|\theta|$ as the \emph{argument} of $a$ and will denote it by $\arg(a)$. According to this convention, $\arg( a )\in [0, \pi]$, unless specified otherwise.

\medskip Suppose a quaternion $q$ is conjugate to a complex number $z=re^{ i \alpha}$. Since $\Re(q)=\Re(z)$ and $|q|=|z|$, it follows that $|\Im q|=|\Im z|=|r\sin \alpha|$, i.e. $|\sin \alpha|=\frac{|\Im q|}{|q|}$.

\subsection{Matrices over the quaternions}
Let  $ {\rm M{(2, \H)}}$ denotes the set of all $2\times2$ matrices over the quaternions. If $A=\begin{pmatrix}a&b\\c&d\end{pmatrix}$, then we can associate the `quaternionic determinant'
$\det(A)=|ad-aca^{-1}b|$.  A matrix $A\in{\rm M{(2, \H)}}$ is invertible if and only if $\det(A)\neq0$.  Also, note that for  $A,B\in {\rm M{(2, \H)}}, ~ \det(AB)=\det(A)\det(B)$.
Now set
$$\S=\bigg\{\begin{pmatrix}a&b\\c&d\end{pmatrix}\in {\rm  M}_2(\H):\det{\begin{pmatrix}a&b\\c&d\end{pmatrix}}       =|ad-aca^{-1}b|=1\bigg \}.$$
The group $\S$ acts as the orientation-preserving isometry group of the hyperbolic $5$-space $\h^5$. We identify the extended quaternionic plane $\hat \H=\H \cup \infty$ with the conformal boundary $\s^4$ of the hyperbolic $5$-space. The group $\S$ acts on $\hat \H$ by M\"obius transformations:
$$\begin{pmatrix}a&b\\c&d\end{pmatrix}: Z \mapsto (aZ+b)(cZ+d)^{-1}.$$
The action is extended over $\h^5$ by Poincar\'e extensions.

\subsection{Classification of elements of $\S$}   Every element $A$ of $\S$ has a fixed point on the closure of the hyperbolic space ${\overline \h}^5$ and this gives us the usual classification of elliptic, parabolic and hyperbolic (or loxodromic) elements in $\S$. Further, it follows from Lefschetz fixed point theorem that every element of $\S$ has a fixed point in conformal boundary. Up to conjugacy, we can take that fixed point to be $\infty$ and hence every element in $\S$ is conjugate to an upper-triangular matrix.

We would like to note here that an elliptic or hyperbolic  element $A$  is conjugate to a matrix of the form $$\begin{pmatrix} \lambda & 0 \\ 0 & \mu \end{pmatrix}$$
where $\lambda, \mu \in \C$. If $|\lambda|=|\mu|(=1)$ then $A$ is elliptic. Otherwise it is hyperbolic. In the hyperbolic case $|\lambda| \neq 1 \neq |\mu|$ and $|\lambda||\mu|=1$.   A hyperbolic or loxodromic element will be called \emph{strictly hyperbolic} if it is conjugate to a real diagonal (non-identity) matrix.
A parabolic isometry is conjugate to an element of the form
$$\begin{pmatrix} \lambda & 1 \\ 0 & \lambda \end{pmatrix}, ~ |\lambda|=1.$$
For more details of the classification and algebraic criteria to detect them see \cite{cao, kg, p, ps}.
\subsection{Conjugacy invariants}
According to Foreman \cite{foreman} the following three functions are conjugacy invariants of $\S$:  For $A=\begin{pmatrix}a&b\\c&d\end{pmatrix}\in\S$,
\begin{eqnarray*}
\beta={\beta}_A&=&|d|^{2}\Re(a)+|a|^{2}\Re(d)-\Re(\overline{a}bc)-\Re(bc\overline{d})\\&=&\Re[(ad-bc)\overline{a}+(da-cb)\overline{d}],\\
\gamma={\gamma}_A&=&|a|^{2}+|d|^{2}+4\Re(a)\Re(d)-2\Re(bc)\\&=&|a|^{2}+|d|^{2}+2[\Re(a\overline{d})+\Re(ad)]-2\Re(bc)\\&=&|a+d|^{2}+2\Re(ad-bc),\\
\delta={\delta}_A&=&\Re(a)+\Re(d)=\Re(a+d)
\end{eqnarray*}
Parker and Short \cite{ps} defined another two quantities for each $A\in \S$ as follows:
\begin{eqnarray*}
\sigma={\sigma}_A
&=&cac^{-1}d-cb,when \; c\neq 0,\\
&=&bdb^{-1}a,when \; c=0,b\neq 0,\\
&=&(d-a)a(d-a)^{-1}d,when \; b=c=0,a\neq d,\\
&=&a\overline{a},when \; b=c=0,a=d\\
\tau={\tau}_A&=&cac^{-1}+d,when \; c\neq 0\\
&=&bdb^{-1}+a,when \; c=0,b\neq 0\\
&=&(d-a)a(d-a)^{-1}+d,when \; b=c=0,a\neq d\\
&=&a+\overline{a},when \; b=c=0,a=d
\end{eqnarray*}
It can be proved that in each case $|\sigma|^{2}=\alpha=1$,where
$$\alpha={\alpha}_A=|a|^2|d|^2+|b|^2|c|^2-2\Re(a\overline{c}d\overline{b}).$$ We are going to show that $\sqrt{\alpha}=det(A)=|ad-aca^{-1}b|=|\sigma|$.
\begin{lemma}If $A=\begin{pmatrix}a&b\\c&d\end{pmatrix}\in{\rm M{(2, \H)}}$, then  $\sqrt{\alpha}=det(A)=|ad-aca^{-1}b|=|\sigma|$.
\end{lemma}
\begin{proof}
We observe that\begin{eqnarray*} (det(A))^2&=&|ad-aca^{-1}b|^2=(ad-aca^{-1}b)\overline{(ad-aca^{-1}b)}\\
&=&(ad-aca^{-1}b)(\overline{d}\overline{a}-\overline{b}{\overline{a}}^{-1}\overline{c}\overline{a})\\
&=&|a|^2|d|^2+|b|^2|c|^2-ad\overline{b}{\overline{a}}^{-1}\overline{c}\overline{a}-aca^{-1}b\overline{d}\overline{a}\\
&=&|a|^2|d|^2+|b|^2|c|^2-2\Re(aca^{-1}b\overline{d}\overline{a})\\
&=&|a|^2|d|^2+|b|^2|c|^2-2\Re(c\overline{a}b\overline{d})=|a|^2|d|^2+|b|^2|c|^2-2\Re(a\overline{c}d\overline{b})=\alpha.
\end{eqnarray*}
This completes the proof.
\end{proof}
\subsection{Some Observations}
It can be checked that ${\alpha}={\alpha}_A=|l_{ij}|^2=|r_{ij}|^2, 1 \leq i,j \leq 2$, where $l_{ij}$, $r_{ij}$ are defined as follows:
\begin{align*}
l_{11}&=da-dbd^{-1}c & l_{12}&=bdb^{-1}a-bc \\
l_{21}&=cac^{-1}d-cb &  l_{22}&=ad-aca^{-1}b\\
r_{11}&=ad-bd^{-1}cd & r_{12}&=db^{-1}ab-cb\\
r_{21}&=ac^{-1}dc-bc & r_{22}&=da-ca^{-1}ba
\end{align*}
\begin{theorem} \cite{kel} Let $M=\begin{pmatrix}a&b\\c&d\end{pmatrix}\in{\rm M{(2, \H)}}$ be such that $det(M)\neq 0$.Then
$M$ is invertible
$$M^{-1}=\begin{pmatrix}{l_{11}}^{-1}d&-{l_{12}}^{-1}b\\-{l_{21}}^{-1}c&{l_{22}}^{-1}a\end{pmatrix}=\begin{pmatrix}d{r_{11}}^{-1}&-b{r_{12}}^{-1}\\-c{r_{21}}^{-1}&a{r_{22}}^{-1}\end{pmatrix}.$$
\end{theorem}
\subsection{ Notations}
For our convenience we use the following notations:
\begin{align*}
d\sptilde&=l_{11}^{-1}d, & c\sptilde&=l_{21}^{-1}c,& b\sptilde&=l_{12}^{-1}b,& a\sptilde&=l_{22}^{-1}a\\
d_{\sptilde}&=dr_{11}^{-1},& c_{\sptilde}&=cr_{21}^{-1},& b_{\sptilde}&=br_{12}^{-1},& a_{\sptilde}&=ar_{22}^{-1}
\end{align*}\\
Kellerhals has proved some interesting properties of these numbers given by following lemma:
\begin{lemma} \cite{kel}
Let $M=\begin{pmatrix}a&b\\c&d\end{pmatrix}\in {\rm M{(2, \H)}}$ be invertible.Then we have the following properties:
\begin{enumerate}
\item $ad_{\sptilde}-bc_{\sptilde}=1=da_{\sptilde}-cb_{\sptilde},~ {d\sptilde}a-{b\sptilde}c=1={a\sptilde}d-{c\sptilde}b$.
\item $a{d\sptilde}-b{c\sptilde}=1=d{a\sptilde}-c{b\sptilde}, ~{d_{\sptilde}}a-{b_{\sptilde}}c=1={a_{\sptilde}}d-{c_{\sptilde}}b$.
\item $a{b\sptilde}=b{a\sptilde},c{d\sptilde}=d{c\sptilde},~{a\sptilde}c={c\sptilde}a,{b\sptilde}d={d\sptilde}b$.
\item $ab_{\sptilde}=ba_{\sptilde},cd_{\sptilde}=dc_{\sptilde},~a_{\sptilde}c=c_{\sptilde}a,b_{\sptilde}d=
d_{\sptilde}b$.
\end{enumerate}
\end{lemma}

\section{J\o{}rgensen inequality for $\S$}\label{jse}
The following proposition gives a J\o{}rgensen inequality for a two-generator subgroup of $\S$ when one of the generators is semisimple.
\begin{theorem}\label{jss}Let $ S=\begin{pmatrix}a&b\\c&d\end{pmatrix}$ and $T=\begin{pmatrix}{\lambda}&0\\0&{\mu}\end{pmatrix}$, $\lambda$ is not similar to $\mu$,  generate a discrete non-elementary subgroup  of $\S$. Then
 $$\{(\Re\lambda -\Re\mu)^2 +(|\Im \lambda|+|\Im\mu|)^2\}(1+|bc|)\geq 1.$$
\end{theorem}

\begin{proof}
Let us suppose that $K_0=\{(\Re\lambda -\Re\mu)^2 +(|\Im \lambda|+|\Im\mu|)^2\}(1+|bc|)<1.$\\
Consider the Shimizu-Leutbecher sequence defined inductively by $$S_0=\begin{pmatrix}a_0&b_0\\c_0&d_0\end{pmatrix}=S=\begin{pmatrix}a&b\\c&d\end{pmatrix}, ~S_{n+1}=\begin{pmatrix}a_{n+1}&b_{n+1}\\c_{n+1}&d_{n+1}\end{pmatrix}=S_nTS_n^{-1}.$$
Now, \begin{eqnarray}\label{sls}S_{n+1}&=&S_nTS_n^{-1}=\begin{pmatrix}a_n&b_n\\c_n&d_n\end{pmatrix}\begin{pmatrix}\lambda&0\\0&\mu\end{pmatrix}
\begin{pmatrix}d\sptilde_n&-b\sptilde_n\\-c\sptilde_n&a\sptilde_n\end{pmatrix}\\
&=&\begin{pmatrix}a_n\lambda&b_n\lambda\\c_n\lambda&d_n\lambda\end{pmatrix}\begin{pmatrix}d\sptilde_n&-b\sptilde_n\\-c\sptilde_n&a\sptilde_n\end{pmatrix}\\
&=&\begin{pmatrix}a_n\lambda d\sptilde_n-b_n\mu c\sptilde_n&-a_n\lambda b\sptilde_n+b_n\mu a\sptilde_n\\
c_n\lambda d\sptilde_n-d_n\mu c\sptilde_n&-c_n\lambda b\sptilde_n+d_n\mu a\sptilde_n\end{pmatrix}\\
&=&\begin{pmatrix}a_{n+1}&b_{n+1}\\c_{n+1}&d_{n+1}\end{pmatrix}
\end{eqnarray}\\
So,\begin{align*}
a_{n+1}&=a_n\lambda d\sptilde_n-b_n\mu c\sptilde_n,& b_{n+1}&=-a_n\lambda b\sptilde_n+b_n\mu a\sptilde_n\\
c_{n+1}&=c_n\lambda d\sptilde_n-d_n\mu c\sptilde_n,& d_{n+1}&=-c_n\lambda b\sptilde_n+d_n\mu a\sptilde_n
\end{align*}\\
Now, we have \begin{eqnarray*}
|b_{n+1}||c_{n+1}|&=&|(-a_n\lambda b\sptilde_n+b_n\mu a\sptilde_n)(c_n\lambda d\sptilde_n-d_n\mu c\sptilde_n)|\\
&=&|a_nb_nc_nd_n||\lambda-a_n^{-1}b_n\mu a\sptilde_n{b\sptilde_n}^{-1}||\lambda-c_n^{-1}d_n\mu c\sptilde_n{d\sptilde_n}^{-1}|
\end{eqnarray*}
By an easy computation, we see that
\begin{eqnarray*}
|\lambda-a_n^{-1}b_n\mu a\sptilde_n{b\sptilde_n}^{-1}|&=&|\Re\lambda+\Im \lambda-\Re\mu-a_n^{-1}b_n(\Im\mu)a\sptilde_n{b\sptilde_n}^{-1}|, \hbox{ since }\; a_nb\sptilde_n=b_na\sptilde_n\\
&=&|(\Re\lambda -\Re\mu)+\Im \lambda -a_n^{-1}b_n(\Im\mu)a\sptilde_n{b\sptilde_n}^{-1}|\\
&=&\sqrt{(\Re\lambda -\Re\mu)^2 +|\Im \lambda -a_n^{-1}b_n(\Im\mu)a\sptilde_n{b\sptilde_n}^{-1}|^2}\\
&\leq&\sqrt{(\Re\lambda -\Re\mu)^2 +(|\Im \lambda|+|\Im\mu|)^2}.\end{eqnarray*}
Similarly, we may deduce that $|\lambda-c_n^{-1}d_n\mu c\sptilde_n{d\sptilde_n}^{-1}|\leq\sqrt{(\Re\lambda -\Re\mu)^2 +(|\Im \lambda|+|\Im\mu|)^2}$.\\
Therefore,
\begin{equation}
|b_{n+1}||c_{n+1}|\leq |a_nb_nc_nd_n|\{(\Re\lambda -\Re\mu)^2 +(|\Im \lambda|+|\Im\mu|)^2\}\end{equation}
Since $|a_nd_n|\leq 1+|b_nc_n|$, this implies
\begin{equation}\label{ine1}
|b_{n+1}||c_{n+1}| \leq  \{(\Re\lambda -\Re\mu)^2 +(|\Im \lambda|+|\Im\mu|)^2\}(1+|b_nc_n|)|b_nc_n|.
\end{equation}
Since, $K_0=\{(\Re\lambda -\Re\mu)^2 +(|\Im \lambda|+|\Im\mu|)^2\}(1+|bc|)<1$, by using induction process we have the relation, $|b_{n+1}c_{n+1}|\leq K_0^n|bc|\Rightarrow b_nc\sptilde_n\rightarrow 0,\;as\; n\rightarrow \infty$, and so, $a_nd\sptilde_n=1+b_nc\sptilde_n\rightarrow 1,\;as\; n\rightarrow\infty$.\\
Since, $|a_{n+1}|=|a_n\lambda d\sptilde_n -b_n\mu c\sptilde_n|,\;|d_{n+1}|=|-c_n\lambda b\sptilde_n +d_n\mu a\sptilde_n|$,  we have\\
$|\lambda||a_nd\sptilde_n|-|\mu||b_nc\sptilde_n|\leq|a_{n+1}|\leq|\lambda||a_nd\sptilde_n|+|\mu||b_nc\sptilde_n| \Rightarrow |a_{n+1}|\rightarrow |\lambda|,\;as\; n\rightarrow\infty$.\\
Similarly, we have $|d_{n+1}| \rightarrow |\mu|,\;\hbox{ as }\; n\rightarrow\infty$.\\
Again,we have
 \begin{eqnarray*}|b_{n+1}|&=&|-a_n\lambda b\sptilde_n +b_n\mu a\sptilde_n|=|a_nb\sptilde_n||\lambda-a_n^{-1}b_n\mu a\sptilde_n{b\sptilde_n}^{-1}|\\&\leq&|a_nb_n|\sqrt{(\Re\lambda -\Re\mu)^2 +(|\Im \lambda|+|\Im\mu|)^2}\\&\leq&{K_0}|a_n||b_n|\rightarrow {K_0}|b_n|,\;\hbox{ since }\; |a_n|\rightarrow 1.\\&\leq&{K_0^n}|b| \rightarrow 0,\;\hbox{ since }\; K_0<1.
\end{eqnarray*}
Thus, for all positive integers, $|b_n|\rightarrow 0\;\hbox{ as }\; n\rightarrow \infty$, i.e. $b_n\rightarrow 0 \;\hbox{ as }\; n\rightarrow \infty$.\\
Similarly, we may show that $c_n \rightarrow 0 \;\hbox{ as }\; n \rightarrow \infty$.\\
Thus the sequence {$S_n$} has a convergent subsequence and
since the subgroup $\langle{A,B}\rangle$ is discrete, so we arrive at a contradiction. This proves the theorem. \end{proof}
\begin{cor}\label{kele}
Let $ S=\begin{pmatrix}a&b\\c&d\end{pmatrix}$ and $T=\begin{pmatrix}{\lambda}&0\\0&{\mu}\end{pmatrix}\in\S$, $\lambda$ is not similar to $\mu$,  generate a discrete non-elementary subgroup $\langle S, T \rangle$  of $\S$. Then
 $$2(\cosh{\tau}-\cos(\alpha+\beta))(1+|bc|)\geq 1,$$ where $\alpha=arg(\lambda),~\beta=arg(\mu)$, $\tau=2 \log |\lambda|$.
\end{cor}
\begin{proof}Without loss of generality, assume $|\lambda|=r\geq 1$. Observe that,
\begin{eqnarray*}
& &(\Re\lambda -\Re\mu)^2 +(|\Im \lambda|+|\Im\mu|)^2\\
&=&(r \cos\alpha -\frac{1}{r} \cos\beta)^2 +(r|\sin\alpha|+\frac{1}{r}|\sin\beta|)^2\\
&=&r^2+\frac{1}{r^2}-2(\cos\alpha \cos\beta -|\sin\alpha||\sin\beta|\big)\\
&=&2(\cosh \tau-\cos(\alpha+\beta)), \hbox{ where } ~ r=e^{\frac{\tau}{2}}, ~ \tau \geq 0.
\end{eqnarray*}
This completes the proof.
\end{proof}

\begin{remark} \label{jre1} 
Kellerhals \cite[Proposition 3]{kel2} proved the above result assuming $T$ hyperbolic, i.e. when $\tau \neq 0$. However, it follows from above that Kellerhals's result carry over to the elliptic case as well, i.e. when $\tau=0$.  Also we have avoided the normalization of the constant term of $(1+|bc|)$ in the inequality to make it sharp. The \thmref{jss} also extends Waterman's Theorem 9  in \cite{waterman} when restricted to the quaternionic set up. Note that $\S$ is not a Clifford group and hence, Theorem 9 of Waterman does not restrict to $\S$. For example, the element 
$$T=\begin{pmatrix} e^{i \theta} & 0 \\ 0 & e^{i \phi} \end{pmatrix},$$
does not belong the Clifford group ${\rm SL}_2(C_2)$, see \cite[p. 95]{waterman}, but it  belongs to the group  $\S$. This class of elements are also covered by \thmref{jss}. 
\end{remark} 

The next theorem generalizes the J\o{}rgensen's inequality in $\S$ for strictly \hbox{hyperbolic} elements with some given conditions. The formulation resembles the original inequality by J\o{}rgensen.
\begin{cor}\label{jh}
Let $A,B\in \S$ be such that both $A$ and the commutator $[A, B]$ are strictly hyperbolic. If $\langle A,B\rangle$ is a non-elementary discrete subgroup of $\S$, then
$$|\delta^2_A-4|+|\delta_{ABA^{-1}B^{-1}}-2|\geq 1.$$
\end{cor}
\begin{proof}
Let $A=\begin{pmatrix}k&0\\0&k^{-1}\end{pmatrix}$,  where $k>1$ and $B=\begin{pmatrix}a&b\\c&d\end{pmatrix}$ with $c\neq 0$.\\
So, $\delta_A=k+k^{-1}\Rightarrow |\delta^2_A-4|=|(k+k^{-1})^2-4|=|k-k^{-1}|^2$.\\
Now,\begin{eqnarray*}
AB&=&\begin{pmatrix}k&0\\0&k^{-1}\end{pmatrix}\begin{pmatrix}a&b\\c&d\end{pmatrix}=\begin{pmatrix}ka&kb\\k^{-1}c&k^{-1}d\end{pmatrix}\\
ABA^{-1}B^{-1}&=&\begin{pmatrix}ka&kb\\k^{-1}c&k^{-1}d\end{pmatrix}\begin{pmatrix}k^{-1}&0\\0&k\end{pmatrix}\begin{pmatrix}c^{-1}d\sigma^{-1}c&-a^{-1}b\sigma^{-1}cac^{-1}\\-\sigma^{-1}c&\sigma^{-1}cac^{-1}\end{pmatrix}\\
&=&\begin{pmatrix}a&k^2b\\k^{-2}c&d\end{pmatrix}\begin{pmatrix}c^{-1}d\sigma^{-1}c&-a^{-1}b\sigma^{-1}cac^{-1}\\-\sigma^{-1}c&\sigma^{-1}cac^{-1}\end{pmatrix}\\
&=&\begin{pmatrix}ac^{-1}d\sigma^{-1}c-k^2b\sigma^{-1}c&(k^2-1)b\sigma^{-1}cac^{-1}\\(k^{-2}-1)d\sigma^{-1}c&d\sigma^{-1}cac^{-1}-k^{-2}ca^{-1}b\sigma^{-1}cac^{-1}\end{pmatrix}.
\end{eqnarray*}
So, we have,
\begin{eqnarray*}\delta_{ABA^{-1}B^{-1}}&=&\Re(ac^{-1}d\sigma^{-1}c-k^2b\sigma^{-1}c)+\Re(d\sigma^{-1}cac^{-1}-k^{-2}ca^{-1}b\sigma^{-1}cac^{-1})\\
&=&\Re(ac^{-1}d\sigma^{-1}c)-k^2\Re(b\sigma^{-1}c)+\Re(d\sigma^{-1}cac^{-1})-k^{-2}\Re(ca^{-1}b\sigma^{-1}cac^{-1})\\
&=&2\Re(cac^{-1}d\overline{\sigma})-(k^2+k^{-2})\Re(b\overline{\sigma}c)\\
&=&2(1+\Re(cb\overline{\sigma}))-(k^2+k^{-2})\Re(b\overline{\sigma}c),  \hbox{ since }\; \sigma=cac^{-1}d-cb.\\
&=&2-(k^2+k^{-2}-2)\Re(b\overline{\sigma}c)\\
&=&2-(k-k^{-1})^2\Re(b\overline{\sigma}c).
\end{eqnarray*}
This implies that $|\delta_{ABA^{-1}B^{-1}}-2|=|k-k^{-1}|^2|\Re(b\overline{\sigma}c)|$. Since $ABA^{-1}B^{-1}$ is strictly hyperbolic, we have
 $$b\overline{\sigma}c=b\overline{d}c\overline{a}-|bc|^2\Rightarrow \Re(b\overline{\sigma}c)=\Re(b\overline{d}c\overline{a})-|bc|^2=\Re(a\overline{c}d\overline{b})-|bc|^2.$$
Also, we have  $b\overline{\sigma}cac^{-1}=0\Rightarrow b\overline{(cac^{-1}d-cb)}cac^{-1}=0\Rightarrow |b|^2(|ad|^2-\overline{b}a\overline{c}d)=0\Rightarrow |ad|^2=\overline{b}a\overline{c}d \hbox{ since} \; bc\neq 0$, for  otherwise $\langle A,B \rangle $ becomes elementary. This shows that $b\overline{\sigma}c=|ad|^2-|bc|^2=\Re(b\overline{\sigma}c)$. Thus,
$$|\delta^2_A-4|+|\delta_{ABA^{-1}B^{-1}}-2|=|k-k^{-1}|^2(1+|bc|).$$
Now the theorem follows from \thmref{ext1}.
\end{proof}

The following two corollaries give weaker versions of \thmref{jss}. 
\begin{cor}\label{jss2}
Suppose $S=\begin{pmatrix} a&b\\ c&d \end{pmatrix}$ and  $T=\begin{pmatrix} \lambda& 0\\ 0& \mu\end{pmatrix}$  generate a non-elementary discrete subgroup in $\S$. Then we have
$$\beta (T) L^k \geq 1,$$
where $$\beta (T) = \displaystyle{\sup_{e,f \neq 0,\infty} |(\lambda -e \mu e^{-1})(\lambda - f \mu f^{-1})|},$$
$$L = 1+|\mu| \;\;\text{and} \;\;k= [1+|bc|] +1,$$
$[~.~]$ denotes the greatest integer function .
\end{cor}

\begin{proof} Since $L>1$, $k>2$, note that $1+|bc| \leq k \leq L^k$.  Let 
$$K=(\Re\lambda -\Re\mu)^2 +(|\Im \lambda|+|\Im\mu|)^2.$$
Using conjugation if necessary, suppose without loss of generality that $\lambda$, $\mu$ are complex numbers. Note that, both $\beta(T)$ and $K$ are invariant if we conjugate the matrix $T$ in the above theorems to a diagonal matrix in $\S$ over the complex numbers. 
Note that 
$$|\lambda-j \mu j^{-1}|= \sqrt{ (\Re \lambda - \Re \mu)^2 + (|\Im \lambda|^2 + |\Im \mu|^2)},$$
hence $K \leq \beta(T)$. 

Further note that a diagonal element $T \in \S$ can be conjugated to a diagonal matrix $T' \in {\rm SL}(2, \C)$ and the conjugation is done using a diagonal element in $\S$. So, given $\langle S, T \rangle$ as in the above results, if we conjugate it to $\langle DSD^{-1}, DTD^{-1} \rangle$, where $D$ a diagonal matrix in $\S$, then the conjugation makes $DTD^{-1}$ a diagonal matrix over $\C$. And further, it is easy to check that if $DTD^{-1}=\begin{pmatrix} a' & b ' \\ c' & d' \end{pmatrix}$, then $|b' c'|=|bc|$. Thus conjugation of $\langle S, T \rangle$ by a diagonal matrix in $\S$ does not change the left hand sides of the above inequalities. 

Since $\langle S, T \rangle$ is discrete, non-elementary, $K(1+|bc|)>1$. Hence 
$$\beta(T)L^k \geq K(1+|bc|) \geq 1.$$
\end{proof}
\begin{cor}\label{jssc2}
Suppose $S=\begin{pmatrix} a&b\\ c&d \end{pmatrix}$ and  $T=\begin{pmatrix} \lambda& 0\\ 0& \mu\end{pmatrix}$  generate a non-elementary discrete subgroup in $\S$. Then we have
$$\beta (T)(1+|bc|)\geq1,$$
where $\beta (T) = \displaystyle{\sup_{e,f \neq 0,\infty} |(\lambda -e \mu e^{-1})(\lambda - f \mu f^{-1})|}$.
\end{cor}

\medskip The next theorem gives J\o{}rgensen inequality for a two-generator subgroup where one of the generators fixes $\infty$.
\begin{theorem}\label{jg}
Suppose $S=\begin{pmatrix} a&b\\ c&d \end{pmatrix}$ and $T=\begin{pmatrix} \lambda& \eta \\ 0& \mu \end{pmatrix}$, where $\Re \lambda= \Re \mu \neq 0$,  $|\lambda|\leq 1\leq |\mu|$, generate a non-elementary discrete subgroup in $\S$. Suppose \\ \hbox{$S(\lambda , \mu) = |\mu| (|\Im \lambda| + | \Im \mu| ) \leq \frac{1}{4\sqrt 2}$}.
Then we have
$$|c| \sqrt{|\tau_0| |t_0|} \geq  \frac{1+\sqrt{1-4\sqrt  2 S(\lambda ,\mu)}}{2},$$
where  $\tau_0 = \lambda(- c^{-1}d) +\eta +(c^{-1}d)  \mu\;\hbox{and, }t_0 = \lambda(ac^{-1}) +\eta - (ac^{-1})  \mu$.
\end{theorem}

\begin{proof}
Let $\alpha=\arg \lambda,~\beta = \arg \mu$.
Denote $r = |\lambda|$, where we see that $r^2 \cos \alpha= \cos \beta$. Consider the Shimizu-Leutbecher sequence
$$S_0 =S, ~S_{n+1}=S_n T {S_n}^{-1}, \; \hbox{ where }S_n =\begin{pmatrix}a_n &b_n\\ c_n &d_n\end{pmatrix}.$$
Now,
\begin{eqnarray*}
S_{n+1} &=& S_n T {S_n}^{-1}\\
&=& \begin{pmatrix} a_n&b_n\\c_n &d_n\end{pmatrix} \begin{pmatrix} \lambda & \eta\\0& \mu\end{pmatrix} \begin{pmatrix}d\sptilde_n&-b\sptilde_n\\-c\sptilde_n&a\sptilde_n\end{pmatrix}\\
&=&\begin{pmatrix}a_n\lambda d\sptilde_n- a_n \eta c\sptilde_n-b_n\mu c\sptilde_n&-a_n\lambda b\sptilde_n+ a_n \eta a\sptilde_n +b_n\mu a\sptilde_n\\
c_n\lambda d\sptilde_n - c_n \eta c\sptilde_n - d_n \mu c\sptilde_n&-c_n\lambda b\sptilde_n+ c_n  \eta a\sptilde_n+d_n\mu a\sptilde_n. \end{pmatrix}
\end{eqnarray*}
Define $\tau_n, t_n$ by
\begin{eqnarray}
 \tau_n  &=& {\lambda} (-{c_n}^{-1} d_n)  + \eta + ({c_n}^{-1} d_n)\mu\\
 t_n &=&  \lambda(a_n {c_n}^{-1}) + \eta-( a_n {c_n}^{-1})  \mu.
\end{eqnarray}
Since $\Re \lambda=\Re\mu$ by assumption, using this we obtain
\begin{eqnarray}
 \tau_n  &=& \Im{\lambda} (-{c_n}^{-1} d_n)  + \eta + ({c_n}^{-1} d_n)\Im \mu\\
 t_n &=& \Im \lambda(a_n {c_n}^{-1}) + \eta-( a_n {c_n}^{-1}) \Im \mu.
\end{eqnarray}
We see that
\begin{eqnarray*}
c_{n+1} &=& c_n\lambda d\sptilde_n-c_n c\sptilde_n - d_n\lambda c\sptilde_n\\
&=& c_n (\Im \lambda (d\sptilde_n{c\sptilde_n}^{-1}) - \eta - ({c_n}^{-1}d_n) \Im \mu)c\sptilde_n\\
&=&  -c_n \{\Im \lambda(-{c_n}^{-1}d_n)  + \eta + ( c_n^{-1}d_n)\Im \mu \} c\sptilde_n\\
 &=& - c_n \tau_n  c\sptilde_n \\
&\Rightarrow&  |c_{n+1}| = |\tau_n c_n| |c_n|.
\end{eqnarray*}
\begin{eqnarray*}
d_{n+1} &=& -c_n\lambda b\sptilde_n+ c_n  \eta a\sptilde_n+d_n\mu a\sptilde_n \\
&=& \Re  \lambda (d_n a\sptilde_n - c_n b\sptilde_n) + c_n \{ \Im \lambda (- b\sptilde_n {a\sptilde_n}^{-1}) + \eta +({c_n}^{-1} d_n) \Im \mu\}a\sptilde_n\\
&=& \Re  \lambda + c_n \{ \Im \lambda ({a_n}^{-1}{c\sptilde_n}^{-1} - {c_n}^{-1}d_n) +\eta + ({c_n}^{-1} d_n) \Im \mu\}a\sptilde_n \\&=& r \cos \alpha + c_n \tau_n a\sptilde_n + c_n \Im \lambda~~{ a_n}^{-1}{c\sptilde_n}^{-1} a\sptilde_n \end{eqnarray*}
By similar computations, we have
$$ a_{n+1} = r \cos \alpha - a_n \tau_n c\sptilde_n + {c\sptilde_n}^{-1} \Im \mu ~c\sptilde_n.$$
Using above equalities, we see that
\begin{eqnarray*}
\tau_{n+1} &=& \Im \lambda (- {c_{n+1}^{-1}}d_{n+1}) + \eta + ({c_{n+1}}^{-1} d_{n+1}) \Im \mu \\
&=& \Im \lambda\{{c\sptilde_n}^{-1} {\tau_n}^{-1} {c_n}^{-1}(r \cos \alpha + c_n \tau_n a\sptilde_n + c_n \Im \lambda~ {a_n}^{-1}{c\sptilde_n}^{-1} a\sptilde_n)\} +\\
&{}& \eta - \{{c\sptilde_n}^{-1} {\tau_n}^{-1} {c_n}^{-1}(r \cos \alpha + c_n \tau_n a\sptilde_n + c_n \Im \lambda~ {a_n}^{-1}{c\sptilde_n}^{-1} a\sptilde_n)\} \Im \mu\\
&=& t_n + r \cos \alpha \;\Im \lambda~ {c\sptilde_n}^{-1}{\tau_n}^{-1}{c_n}^{-1} + \Im \lambda~ ~{c\sptilde_n}^{-1}{\tau_n}^{-1} \Im \lambda~ ~{a_n}^{-1}{c\sptilde_n}^{-1}a\sptilde_n - \\&{}&r \cos \alpha ~{c\sptilde_n}^{-1}{\tau_n}^{-1}{c_n}^{-1} \Im \mu -  {c\sptilde_n}^{-1}{\tau_n}^{-1} \Im \lambda~~ {a_n}^{-1} {c\sptilde_n}^{-1}{a\sptilde_n} \Im \mu \\
\Rightarrow |\tau_{n+1}| &\leq& |t_n| + \frac{(r^2 |\sin \alpha| + |\sin \beta|)(|\cos \alpha| + |\sin \alpha|)}{|\tau_n {c_n}^2|}, \hbox{ using, } |a_n \sptilde|=|a_n||l_{22}^{-1}|,\\
& & |l_{22}^{-1}|=\det S_n =1, \hbox{ and, }~ |\Im \lambda|=|\lambda||\sin \alpha|,~ |\Im \mu|=|\mu||\sin \beta|\\ \\
\Rightarrow |\tau_{n+1}c_{n+1}| &\leq& |\tau_n c_n||t_n c_n| + \sqrt 2 S(\lambda , \mu), ~\hbox{ since, }|\cos \alpha|+|\sin \alpha|\leq \sqrt 2, ~
\text{where, }\end{eqnarray*}
\begin{eqnarray*}
 S(\lambda , \mu) &=& (|\sin \alpha| + |\mu|^2 |\sin \beta|)\\
& = &  (\frac{|\Im \lambda|}{|\lambda|} + |\mu | |\Im \mu|)\\
&=&  |\mu|(|\Im \lambda| + |\Im \mu|).
\end{eqnarray*}
Similarly we also have,
$$|t_{n+1} c_{n+1}| \leq  |\tau_n c_n||t_n c_n| + \sqrt 2 S(\lambda , \mu)$$
In a similar way we can have
\begin{eqnarray*}
|d_{n+1}| &\leq& |\tau_n c_n| |a_n| + 2r\\
|a_{n+1}| &\leq& |\tau_n c_n| |a_n| + \frac{2}{r}\\
\end{eqnarray*}
Also, $|b_{n+1}| \leq |a_n|^2 + r S(\lambda , \mu) |a_n| |b_n|$\\

 Considering the sequence
$$ x_0 = |c| \sqrt {|\tau_0| |t_0|}, ~  x_{n+1} = x_n^2 + \sqrt 2 S(\lambda , \mu).$$
If $ 0 \leq x_0 < \frac{1+\sqrt{1-4 \sqrt 2  S(\lambda , \mu)}}{2} \leq 1$, then $\{x_n\}$ is a monotonically decreasing sequence of real numbers and is bounded above by $ \frac{1+\sqrt{1-4 \sqrt 2 S(\lambda , \mu)}}{2}$ and converges to $\frac{1-\sqrt{1-4 \sqrt 2 S(\lambda , \mu)}}{2}$. Hence
$$|t_n c_n| < \frac{1+\sqrt{1-4 \sqrt 2 S(\lambda , \mu)}}{2}\leq 1, \hbox{ and }$$
$$|\tau_n c_n|<\frac{1+\sqrt{1-4  \sqrt 2 S(\lambda , \mu)}}{2}\leq 1.$$
One a subsequence $|t_n c_n|$ and $|\tau_n c_n|$ converges to values at most $\frac{1-\sqrt{1-4 \sqrt 2 S(\lambda , \mu)}}{2}$. Hence on a subsequence $|a_n|$, $|b_n|$, $|c_n|$, $|d_n|$ converge. In particular, $|c_n| \to 0$. Note that $c_n \neq 0$ unless $c=0$ and $c$ can not be zero as the group $\langle S, T \rangle$ is non-elementary by assumption.

This proves the theorem.
\end{proof}
\begin{cor}
Suppose $S=\begin{pmatrix} a&b\\ c&d \end{pmatrix}$ and $T=\begin{pmatrix} \lambda& \eta \\ 0& \mu \end{pmatrix}$, where $\Re \lambda= \Re \mu\neq 0 $, $\eta \neq 0$,  $|\lambda|\leq 1\leq |\mu|$, \hbox{generate} a non-elementary discrete subgroup in $\S$. Suppose,
$$S'(\lambda , \mu) = \frac {|\mu|}{|\eta|^2}(|\Im \lambda| + | \Im \mu| ).$$
Then we have
$$|c| \sqrt{|\tau'_0| |t'_0|} \geq  \frac{1+\sqrt{1-4 \sqrt 2 |\eta|^2 S'(\lambda ,\mu)}}{2|\eta|},$$
where  $\tau'_0 = \lambda(- c^{-1}d){\eta}^{-1}+ 1 +(c^{-1}d)\mu {\eta}^{-1}\;\hbox{and, }t'_0 = \lambda(ac^{-1}){\eta}^{-1} +1 - (ac^{-1})\mu {\eta}^{-1}$.
\end{cor} 
\begin{proof}
If $\eta \neq 0$, we write $\tau_0= \tau'_0 \eta$ and $t_0=t'_0 \eta$.  Then the result follows from the inequality in \thmref{jg}.
\end{proof}

\begin{cor}\label{rez}
Suppose $S=\begin{pmatrix} a&b\\ c&d \end{pmatrix}$ and $T=\begin{pmatrix} \lambda& \eta \\ 0& \mu \end{pmatrix}$, where $\Re \lambda=\Re \mu=0$, $|\lambda|\leq 1\leq |\mu|$, generate a non-elementary discrete subgroup in $\S$.    Suppose \hbox{$S(\lambda , \mu) =  |\mu| (|\Im \lambda| + | \Im \mu| ) \leq  \frac{1}{4}$}.
Then we have
$$|c| \sqrt{|\tau_0| |t_0|} \geq  \frac{1+\sqrt{1- 4S(\lambda ,\mu)}}{2},$$
where  $\tau_0 = \lambda(- c^{-1}d) +\eta +(c^{-1}d)  \mu\;\hbox{and, }t_0 = \lambda(ac^{-1}) +\eta - (ac^{-1})  \mu$.
\end{cor}
\begin{proof}
In this case, we proceed as in the proof of the previous theorem. The only difference from the previous proof is essentially the following bound:
\begin{eqnarray*}
 |\tau_{n+1}| &\leq& |t_n| + \frac{(r^2 |\sin \alpha| + |\sin \beta|) |\sin \alpha|}{|\tau_n {c_n}^2|}, \hbox{ using, } |a_n \sptilde|=|a_n||l_{22}^{-1}|,\\
& & |l_{22}^{-1}|=\det S_n =1, \hbox{ and, }~ |\Im \lambda|=|\lambda||\sin \alpha|,~ |\Im \mu|=|\mu||\sin \beta|\\ \\
\Rightarrow |\tau_{n+1}c_{n+1}| &\leq& |\tau_n c_n||t_n c_n| + S(\lambda , \mu), ~
\text{where, }\end{eqnarray*}
\begin{eqnarray*}
 S(\lambda , \mu) &=&  (|\sin \alpha| + |\mu|^2 |\sin \beta|)\\
& = & (\frac{|\Im \lambda|}{|\lambda|} + |\mu | |\Im \mu|)\\
&=&  |\mu|(|\Im \lambda| + |\Im \mu|).
\end{eqnarray*}
Noting this bound, the rest is similar.
\end{proof}
Given any parabolic transformation in $\S$, it is conjugate to a transformation of the form $$T=\begin{pmatrix} \lambda& 1\\ 0& \lambda\end{pmatrix}, ~ |\lambda|=1,$$ 
and moreover, one can choose $\Re(\lambda)=0$ up to conjugacy. Thus, using \corref{rez}, gives Waterman's result  \cite[Theorem 8]{waterman} in $\S$. 

\begin{cor}\label{wat}
If  $S=\begin{pmatrix} a&b\\ c&d \end{pmatrix},~T=\begin{pmatrix} \lambda& 1\\ 0& \lambda\end{pmatrix}$, $|\lambda|=1$ generates a non-elementary discrete subgroup in $\S$ with $T$  parabolic fixing $\infty$, then
$$|c| \sqrt{|T(ac^{-1}) - ac^{-1}|}\sqrt{|T(-c^{-1}d) - (-c^{-1}d)|} \geq \frac{1+\sqrt{1-8 |\Im \lambda|}}{2}.$$
\end{cor} 
\begin{proof}
Note that $T(a c^{-1})=(\lambda (a c^{-1})+ 1)\lambda^{-1}$ and
$T(-c^{-1} d)=(\lambda (-c^{-1} d) + 1)\lambda^{-1}$. Now,
$T(ac^{-1})-(ac^{-1})=(\lambda (ac^{-1})+ 1 -(ac^{-1}) \lambda) \lambda^{-1}=t_0 \lambda^{-1}$. Similarly,
$ T(-c^{-1}d) - (-c^{-1}d)=\tau_o \lambda^{-1}$. Since $|\lambda|=1$, the result follows.
\end{proof}

Recently, Erlandsson and Zakeri \cite{ez}  have proved  a more geometric version of  \thmref{jg}. 
Their geometric inequality does not depend on any quantity like $S(\lambda, \mu)$. Also, in the asymptotic case, it covers some of the two-generators groups whose discreteness remain inconclusive by \corref{wat}. However, the inequality of Erlandsson and Zakeri  does not involve the algebraic coefficients of the matrices. In that sense, the theorems in this paper give a more explicit  algorithm involving the matrix coefficients to test discreteness.

\medskip Using similar argument as in the proof of \thmref{jg}, we can also prove the following theorem that  gives J\o{}rgensen inequality for a two-generator subgroup where one of the generators has a fixed point $0$.

\begin{theorem}\label{jlt}
Suppose $S=\begin{pmatrix} a&b\\ c&d \end{pmatrix}$ and $T=\begin{pmatrix} \lambda& 0 \\ \eta& \mu \end{pmatrix}$, where $\Re \lambda= \Re \mu =\kappa$,  $|\lambda|\leq 1\leq |\mu|$, generate a non-elementary discrete subgroup in $\S$. Suppose,   \hbox{$ S(\lambda , \mu) = |\mu|  (|\Im \lambda| + | \Im \mu| ) \leq \epsilon$}. Then
$$|c| \sqrt{|\tau_0| |t_0|} \geq  \frac{1+\sqrt{1-\epsilon^{-1} S(\lambda ,\mu)}}{2},$$
where  $\tau_0 = \mu (- b^{-1} a) +\eta +(b^{-1} a) \lambda ,\;t_0 = \mu (d b^{-1}) +\eta - (d b^{-1}) \lambda$ and $\epsilon=\frac{1}{4 \sqrt 2}$ or $\frac{1}{4}$ depending upon $\kappa \neq 0$ or $\kappa=0$. 
\end{theorem}

\section{Extremality of J\o{}rgensen Inequality} \label{sext}
The following theorem generalizes Theorem-1 of J\o{}rgensen-Kikka \cite{jk}.
\begin{theorem}\label{ext1}
Let $ S=\begin{pmatrix}a&b\\c&d\end{pmatrix}$ and $T=\begin{pmatrix}{\lambda}&0\\0&{\mu}\end{pmatrix} \in\S$. Suppose, $\langle S, T \rangle$ is discrete, non-elementary and  for $\alpha=arg(\lambda),~\beta=arg(\mu)$,  $\tau=2 \log |\lambda|$,
$$\{(\Re\lambda -\Re\mu)^2 +(|\Im \lambda|+|\Im\mu|)^2\}(1+|bc|)= 1.$$
We consider the Shimizu-Leutbechar sequence
$$S_0=S, \hspace{.5 cm} S_{n+1}= S_n T S_n^{-1}.$$
Then $T$ and $S_{n+1}=S_n T S_n^{-1}$ generate a non-elementary discrete group and
$$\{(\Re\lambda -\Re\mu)^2 +(|\Im \lambda|+|\Im\mu|)^2\}(1+|b_n c_n|)= 1.$$
\end{theorem}
\begin{proof}
We consider the Shimizu-Leutbechar sequence
$$S_0=S, \hspace{.5 cm} S_{n+1}= S_n T S_n^{-1}.$$
 From relation (3.1) we get
\begin{align*}
a_{n+1}&=a_n\lambda d\sptilde_n-b_n\mu c\sptilde_n,& b_{n+1}&=-a_n\lambda b\sptilde_n+b_n\mu a\sptilde_n\\
c_{n+1}&=c_n\lambda d\sptilde_n-d_n\mu c\sptilde_n,& d_{n+1}&=-c_n\lambda b\sptilde_n+d_n\mu a\sptilde_n
\end{align*}\\
and we also have,  \begin{eqnarray*}
|b_{n+1}||c_{n+1}|&=&|(-a_n\lambda b\sptilde_n+b_n\mu a\sptilde_n)(c_n\lambda d\sptilde_n-d_n\mu c\sptilde_n)|\\
&=&|a_nb_nc_nd_n||\lambda-a_n^{-1}b_n\mu a\sptilde_n{b\sptilde_n}^{-1}||\lambda-c_n^{-1}d_n\mu c\sptilde_n{d\sptilde_n}^{-1}|.
\end{eqnarray*}
This implies, (see \eqnref{ine1} in the proof of \thmref{jss})
\begin{equation}\label{ee1} |b_{n+1}c_{n+1}| \leq \{(\Re\lambda -\Re\mu)^2 +(|\Im \lambda|+|\Im\mu|)^2\} (1+|b_n c_n|)
| b_n c_n|. \end{equation}
Let
\begin{equation}\label{k} K=(\Re\lambda -\Re\mu)^2 +(|\Im \lambda|+|\Im\mu|)^2. \end{equation}
 Construct the sequence $w_n$ where
$$w_0=|bc|, \hspace{.2in} w_{n}=|b_n c_n|.$$
It follows from \eqnref{ee1} that $w_{n+1} \leq K w_n (1+w_n)$. Now note that $K(1+w_0)=1$. Now $w_0 \neq 0$, for otherwise, $S$ and $T$ will have a common fixed point.  Hence $K<1$.

Observe that
$$1 \leq K(1+w_1) \leq K(1+Kw_0(1+w_0)) \leq K(1+w_0)=1,$$
and hence $K(1+w_1)=1$. By induction it follows that $K(1+w_n)=1$ for all $n \geq 0$.
Since $K<1$, it follow that $w_n \neq 0$ for all $n$ and hence the result follows.
\end{proof}
The following corollary generalizes Theorem-2 of J\o{}rgensen and Kikka \cite{jk}.
\begin{cor}\label{extc1}
Let $ S=\begin{pmatrix}a&b\\c&d\end{pmatrix}$ and $T=\begin{pmatrix}{\lambda}&0\\0&{\mu}\end{pmatrix} \in\S$. If $\langle S, T \rangle$ is discrete, non-elementary and
$$\{(\Re\lambda -\Re\mu)^2 +(|\Im \lambda|+|\Im\mu|)^2\}(1+|bc|)=1,$$
then $T$ is elliptic of order at least seven.
\end{cor}
\begin{proof}
If possible suppose $T$ is hyperbolic.  As in the above proof, it follows from the extremal relation that $K<1$.
Now, Let $\arg \lambda=\alpha$ and $\arg \mu=\beta$. Then
\begin{eqnarray*}
K&=& (\Re\lambda -\Re\mu)^2 +(|\Im \lambda|+|\Im\mu|)^2\\
&=& |\lambda|^2 + |\mu|^2 +2 (|\Im \lambda| | \Im \mu|-\Re \lambda \Re \mu)\\
&=& |\lambda|^2 + |\mu|^2 + 2 |\sin \alpha| |\sin \beta|-\cos \alpha \cos \beta\\
&=& |\lambda|^2 + |\mu|^2-2\cos(\alpha+  \beta).
\end{eqnarray*}
Let $|\lambda|=e^{\frac{\tau}{2}}$. Then using $\cosh(\tau)=\frac{e^{\tau} + e^{-\tau} }{2}$, observe that
\begin{eqnarray*}
K & = & e^{\tau} + e^{-\tau} - 2 \cos(\alpha+ \beta)\\
&\geq&  e^{\tau} + e^{-\tau}+2=(e^{\frac{\tau}{2} }+e^{-\frac{\tau}{2} })^2.
\end{eqnarray*}
Since $e^{\frac{\tau}{2} }+e^{-\frac{\tau}{2} }>1$, this implies $K>1$.  This is a contradiction. Hence $T$ must be elliptic.

Since $T$ is elliptic, $\tau=0$.  Now, $K=1$ implies, $\cos(\alpha+\beta)>\frac{1}{2}$. Thus $0<\alpha+\beta<\frac{\pi}{3}$. This implies that the order of $T$ must be at least seven.

This completes the proof.
\end{proof}
\begin{cor}\label{extcc2}
Let $ S=\begin{pmatrix}a&b\\c&d\end{pmatrix}$ and $T=\begin{pmatrix}{\lambda}&0\\0&{\mu}\end{pmatrix} \in\S$. Suppose, $\langle S, T \rangle$ is discrete, non-elementary and    
$$\beta(T)(1+|bc|) = 1,$$
then $T$ is elliptic of order at least seven.
\end{cor}
\begin{proof}
Suppose, up to conjugacy, $\lambda$, $\mu$ are complex numbers. Then $K \leq \beta(T)$. Since $\langle S, T \rangle$ is discrete, we must have $K(1+|bc|) \geq 1$. Hence the equality in the hypothesis implies $K(1+|bc|)=1$. The result now follows from the above corollary. 
\end{proof}

\begin{cor}\label{extsc3}
Let $ S=\begin{pmatrix}a&b\\c&d\end{pmatrix}$ and $T=\begin{pmatrix}{\lambda}&0\\0&{\mu}\end{pmatrix} \in\S$.  If $\langle S, T \rangle$ is discrete, non-elementary and
$$\beta(T) L^k=1,$$
where $k = [1+|bc|] +1 >2$ and $L= 1+|\mu|>1$,
then $T$ is elliptic of order at least seven.
\end{cor}

\begin{proof}
Up to conjugacy, we assume $\lambda$, $\mu $ are complex numbers. 
It is enough to show that $K(1+|bc|) =1$, where $K=(\Re\lambda -\Re\mu)^2 +(|\Im \lambda|+|\Im\mu|)^2$. Since the subgroup $\langle S,T \rangle$ generates a discrete non-elementary subgroup of $\S$, then we have $K(1+|bc|) \geq 1$. Now note that
 $$K \leq \beta(T)=\frac{1}{L^{k}} \leq \frac{1}{1+|bc|},$$
This imples, $K(1+|bc|) \leq 1$. Hence, $K(1+|bc|)=1$. The result now follows from \corref{extc1}. 
\end{proof}
The following characterizes non-extreme groups.
\begin{cor}\label{extp1}
Let $ S=\begin{pmatrix}a&b\\c&d\end{pmatrix}$ and $T=\begin{pmatrix}{\lambda}&0\\0&{\mu}\end{pmatrix}$ generate a discrete non-elementary subgroup in $\S$. Suppose
$$||ad|-1|> \big({\cot}^2 \big(\frac{\alpha + \beta}{2}\big) - 3\big).$$
Then $\langle S, T \rangle$ is not an extreme group. \end{cor}

\begin{proof} If possible suppose $\langle S, T \rangle$ satisfy equality in J\o{}rgensen inequality. Note that, it follows from the equality in J\o{}rgensen inequality that
\begin{equation} \label{bc} |bc|=\frac{1-K}{K}, \end{equation}
where $ K=(\Re\lambda -\Re\mu)^2 +(|\Im \lambda|+|\Im\mu|)^2$ .
The condition $|\sigma|=|ad-aca^{-1}b|=1$ implies,
\begin{eqnarray*}
1 & \leq & |ad|+|bc| \\
\Rightarrow |ad|& \geq & 1-|bc| \\
&=& K(1+|bc|)-|bc| \\
&=& K+ (K-1) \frac{(1-K)}{K}\\
&=&2-\frac{1}{K}.
\end{eqnarray*}
This implies
\begin{equation} \label{ad} |ad|\geq 1-|bc|. \end{equation}
Also we have from $|\sigma|=1$ that $|ad|-|bc|\leq 1$. This implies
$|ad|\leq 1+|bc|$. Combining this with \eqnref{ad} we get
$$||ad|-1|\leq |bc|.$$
Now we see that $K=2(1-\cos (\alpha + \beta))$ and,
\begin{eqnarray*}
|bc|&=& \frac{1-K}{K}= \frac{2\cos (\alpha + \beta) -1}{2-2\cos (\alpha + \beta)}\\
&=& \frac{\cos^2 ({\frac{\alpha + \beta}{2}})-3 \sin^2 (\frac{\alpha + \beta}{2})}{4\sin^2 (\frac{\alpha + \beta}{2})}\\
&=& \frac{\cot^2 (\frac{\alpha + \beta}{2})-3}{4}\\
&\leq& (\cot^2 (\frac{\alpha + \beta}{2}) -3).\end{eqnarray*}
Hence, we have
$$||ad|-1| \leq (\cot^2 (\frac{\alpha + \beta}{2}) -3),$$
which is a contradiction. This proves the result.
\end{proof}

\medskip 

\begin{theorem}\label{extt2}
Suppose $S=\begin{pmatrix} a&b\\ c&d \end{pmatrix}$ and  $T=\begin{pmatrix} \lambda& \eta \\ 0& \mu\end{pmatrix}$,  $\Re \lambda = \Re \mu=\kappa$, $|\lambda|\leq 1\leq |\mu|$,  generate a non-elementary discrete subgroup in $\S$. Suppose,
$$|c| \sqrt{|\tau_0||t_0|} = \frac{1+\sqrt{1- \epsilon^{-1} S(\lambda , \mu)}}{2},$$
where  $S(\lambda , \mu)=|\mu|(|\Im \lambda|+| \Im \mu| )\leq \epsilon$ and, $\epsilon=\frac{1}{4 \sqrt 2}$ or $\frac{1}{4}$ depending on $\kappa \neq 0$ or $\kappa=0$ . We consider the Shimizu-Leutbechar sequence
$$S_0=S, \hspace{.5 cm} S_{n+1}= S_n T S_n^{-1}.$$
   Then, for each $n$, $\langle S_{n} , T \rangle$ is a non-elementary discrete subgroup of $\S$ and
$$|c_n| \sqrt{|\tau_n||t_n|} = \frac{1+\sqrt{1- \epsilon^{-1} S(\lambda , \mu)}}{2}.$$
where $\tau_n =  \lambda (-{c_n}^{-1}d_n)+\eta + ({c_n}^{-1}d_n)\mu ,\;t_n=\lambda (a_n {c_n}^{-1})+\eta - (a_n{c_n}^{-1}) \mu$. \end{theorem}
\begin{proof} We prove the result assumeing $\kappa \neq 0$. The case $\kappa=0$ is similar. 

Consider the Shimizu-Leutbecher sequence $S_0 =S, ~S_{n+1}=S_n T {S_n}^{-1}$, where $S_n =\begin{pmatrix}a_n &b_n\\ c_n &d_n\end{pmatrix}$\\{}\\
Now,\begin{eqnarray*}
S_{n+1} &=& S_n T {S_n}^{-1}\\
&=& \begin{pmatrix} a_n&b_n\\c_n &d_n\end{pmatrix} \begin{pmatrix} \lambda & \eta\\0& \mu\end{pmatrix} \begin{pmatrix}d\sptilde_n&-b\sptilde_n\\-c\sptilde_n&a\sptilde_n\end{pmatrix}\\
&=&\begin{pmatrix}a_n\lambda d\sptilde_n- a_n \eta c\sptilde_n-b_n\mu c\sptilde_n&-a_n\lambda b\sptilde_n+ a_n \eta a\sptilde_n +b_n\mu a\sptilde_n\\
c_n\lambda d\sptilde_n-c_n \eta c\sptilde_n - d_n\mu c\sptilde_n&-c_n\lambda b\sptilde_n+ c_n \eta a\sptilde_n+d_n \mu a\sptilde_n\end{pmatrix}
\end{eqnarray*}
Define $\tau_n , t_n$ by
\begin{eqnarray}
 {\tau}_n  &=& {\lambda} (-{c_n}^{-1} d_n) +\eta + ({c_n}^{-1} d_n) \mu \\
 t_n &=& \lambda (a_n {c_n}^{-1}) + \eta - (a_n {c_n}^{-1}) \mu
\end{eqnarray}
We see that \begin{eqnarray*}
c_{n+1} &=& c_n\lambda d\sptilde_n-c_n \eta c\sptilde_n - d_n\mu c\sptilde_n\\ &=& - c_n ( \lambda (-d\sptilde_n{c\sptilde_n}^{-1}) +\eta + ({c_n}^{-1}d_n ) \mu)c\sptilde_n\\
&=& - c_n \tau_n  c\sptilde_n\\
\text{So,}\;|c_{n+1}| &=& |\tau_n c_n| |c_n|
\end{eqnarray*}
In a similar way we can have
\begin{eqnarray*}
|d_{n+1}| &\leq& |\tau_n c_n| |a_n| + 2r\\
|a_{n+1}| &\leq& |\tau_n c_n| |a_n| + \frac{2}{r}\\
\end{eqnarray*}
Also, $|b_{n+1}| \leq |a_n|^2 + r S(\lambda , \mu) |a_n||b_n|$, as in the proof of \thmref{jg}. Also, we have

\begin{eqnarray*}
|\tau_{n+1} c_{n+1}| &\leq& |\tau_n c_n| |t_n c_n| + \sqrt 2 S(\lambda , \mu)\\
|t_{n+1} c_{n+1}| &\leq& |\tau_n c_n| |t_n c_n| + \sqrt 2 S(\lambda , \mu)
\end{eqnarray*}
\\ Considering the sequence
$$ x_0 = |c| \sqrt {|\tau_0| |t_0|}, ~  x_{n+1} = x_n^2 + \sqrt 2 S(\lambda , \mu), ~ \hbox{ where, }S(\lambda , \mu) \leq \frac{1}{4\sqrt 2}.$$
Note that $\{x_n\}$ is a monotonically decreasing sequence of real numbers and is bounded above by $ \frac{1+\sqrt{1-4 \sqrt 2 S(\lambda , \mu)}}{2}$.   By the hypothesis $x_0=\frac{1+\sqrt{1-4 \sqrt 2 S(\lambda, \mu)|}}{2}$. Hence $\{x_n\}$ must be a constant sequence. In particular, $c_n \neq 0$ for all $n$ and hence, $S_n$ and $T$ can not have a common fixed point. Thus $\langle S_n, T \rangle$ is non-elementary.
\end{proof}
\begin{cor}\label{extc2}
Let $S=\begin{pmatrix} a&b\\ c&d \end{pmatrix} ,~T=\begin{pmatrix} \lambda& \eta\\ 0& \mu\end{pmatrix}$,\;where $\Re \lambda =\Re \mu$, $|\lambda|\leq 1\leq |\mu|$, generate a discrete, non-elementary subgroup of $\S$. Suppose
\begin{eqnarray*}{\tau}_0  &=& {\lambda} (-{c}^{-1} d) + \eta + ({c}^{-1} d) \mu,\\
 t_0 &=& \lambda (a {c}^{-1}) + \eta - (a {c}^{-1}) \mu. \end{eqnarray*} If
$$\frac{|\tau_0-t_0|}{|\tau_0 t_0|} >  |\bar c d + a \bar c|,$$
then $\langle S, T \rangle$ is not extreme.
\end{cor}
\begin{proof}
Without loss of generality, assume $|\lambda|\leq 1$. Suppose $\langle S, T \rangle$ is extreme. Then $|c|^2 |\tau_0 t_0|={\kappa_0}^2$.
Note that
$$\tau_0-t_0=-\lambda(c^{-1} d + ac^{-1}) + (c^{-1} d + ac^{-1}) \mu=\Im \lambda(c^{-1} d + ac^{-1})+(c^{-1} d + ac^{-1}) \Im \mu.$$
Thus
\begin{eqnarray*} |\tau_0-t_0| &\leq& (|\Im \lambda|+|\Im \mu|)( |c^{-1}d+ac^{-1}|)\\
& \leq &(|\Im \lambda|+|\Im \mu|) {\frac{1} {|c|^2} } | \bar c d + a \bar c|. \frac{|c|^2|\tau_0 t_0|}{{\kappa_0}^2}.
\end{eqnarray*}
This implies
$$\frac{|\tau_0-t_0|}{|\tau_0 t_0|}\leq S(\lambda, \mu) \frac{|\bar c d + a \bar c|}{{\kappa_0}^2}.$$
Now note that $S(\lambda, \mu) \leq \frac{1}{4 \sqrt 2}<\frac{1}{4}$ and $\kappa_0\geq \frac{1}{2}$, hence $\frac{S(\lambda, \mu)}{\kappa_0^2} \leq 1$. So,
$$\frac{|\tau_0-t_0|}{|\tau_0 t_0|} \leq  |\bar c d + a \bar c|.$$
This proves the result.
\end{proof}

If we choose $T=\begin{pmatrix} \lambda& 0 \\ \eta& \mu\end{pmatrix}$, then  analogous results to \thmref{extt2} and \corref{extc2} follow using similar arguments as above. 

\begin{cor}\label{extc3}
Let $S=\begin{pmatrix} a&b\\ c&d \end{pmatrix} ,~T=\begin{pmatrix} \lambda& 0\\ \eta& \mu\end{pmatrix}$,  $\Re \lambda =\Re \mu$, generate a discrete, non-elementary subgroup of $\S$. Suppose
\begin{eqnarray*}\tau_0 &=&\mu (- b^{-1} a) +\eta +(b^{-1} a)  \lambda ,\\t_0 &=&  \mu (d b^{-1}) +\eta - (d b^{-1})  \lambda.\end{eqnarray*} If
$$\frac{|\tau_0-t_0|}{|\tau_0 t_0|} > |\bar b d + a \bar b|,$$
then $\langle S, T \rangle$ is not extreme.
\end{cor}

\subsection{Examples of Extreme Groups}
Let us consider $~S=\begin{pmatrix}a&0\\c&d\end{pmatrix},~T=\begin{pmatrix}\lambda&c^{-1}j\\0&\mu\end{pmatrix} \in \S$ with $|c|\geq 1$ and $\Im\lambda =\Im\mu=0$. Suppose that the subgroup $\langle S,T \rangle$ in $\S$ is non-elementary and discrete. For eg. if $a=d=c=1$ and  $\lambda=\mu=1$ then this is the case.   Then we see that $\tau_0 = c^{-1}j,~ t_0= c^{-1}j$ and so we have $|c|\sqrt{|\tau_0||t_0|} = 1$,whereas we also observe that $S(\lambda , \mu)=0$ and so we have $\frac{1+\sqrt{1-4\sqrt 2 S(\lambda , \mu)}}{2}= \frac{2}{2}= 1\;.$ So, we have $|c|\sqrt{|\tau_0||t_0|} = 1=\frac{1+\sqrt{1-4\sqrt 2 S(\lambda , \mu)}}{2}$.

\end{document}